\documentclass{article}
\usepackage{amsmath,amssymb}
\usepackage{longtable}
\usepackage{fancyhdr}

\usepackage{amsthm}

\newtheorem{theo}{Theorem}[section]

\newtheorem{pr}[theo]{Proposition}
\newtheorem{col}[theo]{Corollary}

\theoremstyle{definition}
\newtheorem{defn}[theo]{Definition}

\newtheorem{ex}[theo]{Example}
\newtheorem{nt}[theo]{Notation}

\def\rank{\mathop{\rm rank}\nolimits}
\def\dim{\mathop{\rm dim}\nolimits}

\def\ad{\mathop{\rm ad}\nolimits}

\def\Tr{\mathop{\rm Tr}\nolimits}

\def\C{\mathop{\mathbb{C}}\nolimits}
\def\Z{\mathop{\mathbb{Z}}\nolimits}

\def\g{\mathop{\mathfrak{g}}\nolimits}

\def\Hom{\mathop{\rm Hom}\nolimits}
\def\sl{\mathop{\mathfrak{sl}}\nolimits}
\def\sp{\mathop{\mathfrak{sp}}\nolimits}
\def\so{\mathop{\mathfrak{so}}\nolimits}
\def\gl{\mathop{\mathfrak{gl}}\nolimits}
\def\l{\mathop{\mathfrak{l}}\nolimits}

\def\U{\mathop{\mathcal{U}}\nolimits}
\def\Lie{\mathop{\mathrm{Lie}}\nolimits}

\newcommand\bigzerou{\smash{\lower.3ex\hbox{\Huge 0}}} 
         
\newcommand\bigstaru{\smash{\lower.3ex\hbox{\Huge $*$}}} 

\numberwithin{equation}{section}

\allowdisplaybreaks[1]
\title{Reduced contragredient Lie algebras and PC Lie algebras}
\vspace{50truept}
\author{ Nagatoshi Sasano}
\begin{document}
\thispagestyle{empty}

\newpage
\begin{center}{\huge  Graded Lie algebras and regular prehomogeneous vector spaces with one-dimensional scalar multiplication \footnote{{\bf 2010 Mathematic Subjects Classification}: Primary 11S90, Secondary 17B70, 17B65}}\end{center}
\vspace{50truept}

\begin{center}{Nagatoshi SASANO}\end{center}
\begin{abstract}
The aim of this paper is to study relations between regular reductive PVs with one-dimensional scalar multiplication and the structure of graded Lie algebras.
We will show that the regularity of such PVs is described by an $\sl _2$-triplet of a graded Lie algebra.
\end{abstract}

\section* {Introduction} 
A prehomogeneous vector space (abbrev. PV) is a triplet $(G,\rho,V)$ consists of a connected algebraic group $G$ and its finite-dimensional rational representation $(\rho ,V)$ such that there exits a Zariski dense orbit.
Some particular cases of PVs are obtained from a graded finite-dimensional semisimple Lie algebra $\mathfrak{l}=\bigoplus _{n\in \Z}\mathfrak{l}_n$ as $(\mathfrak{l}_0,\mathfrak{l}_1)$.
Such spaces are named PVs of parabolic type by H.Rubenthaler (see, for example, \cite {ru}) and studied by him.
Then, today, it is known that PVs of parabolic type have rich structures related to the structure of graded Lie algebras. 
For example, the regularity of irreducible PVs of parabolic type $(\mathfrak{l}_0,\mathfrak{l}_1)$ is closely related to a subalgebra of $\mathfrak{l}$ which is isomorphic to $\sl _2$ (\cite [Corollaire II.2.15]{ru}).
\par 
In \cite {ru}, H.Rubenthaler classified PVs of parabolic type by using Dynkin diagrams of finite-dimensional semisimple Lie algebras.
On the other hand, it is known that there exist infinitely many PVs which are not of parabolic type.
\par 
Recently, the author and H.Rubenthaler independently showed that any finite-dimensional reductive Lie algebra and its finite-dimensional, of course any PV with a reductive group, can be embedded into some (finite or infinite-dimensional) graded Lie algebra (\cite [the author]{Sa3}, \cite [H.Rubenthaler]{Ru}).
Thus, it is expected that we can extend the theory of PVs of parabolic type to the general theory of PVs.
The aim of this paper is to study relations between the regularity of (might not be of parabolic type) PVs with $1$-dimensional scalar multiplication and the structure of graded Lie algebras.
\begin{nt}
\begin{itemize}
\item {For any vector space $W$, we denote the set of all linear maps $f\colon W\rightarrow \C$ by $\Hom (W,\C)$. Moreover, when $W$ is finite-dimensional, we denote by $W^*=\Hom(W,\C)$ simply}.
\item {We denote the zero-matrix of size $k\times l$ by $O_{k,l}$ or $O_k$ when $k=l$, the unit matrix of size $k$ by $I_k$.
We denote the set of all matrices of size $k\times l$ by $M(k,l)$.
}
\item {In this paper, all objects are defined over the complex number field $\C $.}
\end{itemize}
\end{nt}

\section {Graded Lie algebras}
The first aim of this section is to recall some notion and results on graded Lie algebras, prehomogeneous vector spaces, which will be used in the next section.
The second aim is to  summarize the results of the author and of H.Rubenthaler (Theorem \ref {th;stapfund}).
For this, let us start with the definition of standard pentads by the author.
This gives a (finite or infinite-dimensional) graded Lie algebra.
\begin{defn}[\text {\cite [Definitions 2.1 and 2.2]{Sa3}}]
Let $\g $ be a Lie algebra, $\pi $ a representation of $\g $ on $U$, $\U $ a $\g $-submodule of $\Hom (U,\C)$, $B$ a non-degenerate invariant bilinear form on $\g$.
When a pentad $(\g,\pi ,U,\U,B)$ satisfies the following conditions, we say that the pentad $(\g,\pi,U,\U,B)$ is a standard pentad:
\begin{enumerate}
\item {the restriction of the canonical pairing $\langle \cdot ,\cdot \rangle :U\times \Hom (U,\C)\rightarrow \C$ to $U\times \U$ is non-degenerate},
\item {there exists a linear map $\Phi_{\pi }:U\otimes \U \rightarrow \g $, called the $\Phi $-map of the pentad, satisfying an equation
$$
B(a,\Phi _{\pi }(v\otimes \phi ))=\langle \pi (a)v,\phi \rangle
$$}
for any $a\in \g$, $v\in U$ and $\phi \in \U$.
\end{enumerate}
\end{defn}
\begin{theo}[\text {\cite [Theorem 2.15]{Sa3}}]\label {th;staplie}
For any standard pentad $(\g,\pi ,U,\U,B)$, there exists a (finite or infinite dimensional) graded Lie algebra $L(\g,\pi ,U,\U,B)=\bigoplus _{n\in \Z}U_n$ such that 
$$
U_0\simeq \g\text { (as Lie algebras)},\quad U_1\simeq U,\quad U_{-1}\simeq \U\text { (as $U_0\simeq \g $-modules)}
$$
and that the restricted bracket product $[\cdot ,\cdot ]\colon U_1\times U_{-1}\rightarrow U_0$ is induced by the $\Phi $-map $\Phi _{\pi }$ of $(\g,\rho,U,\U,B)$.
\end{theo}
In the sense of Theorem \ref {th;staplie}, we can obtain a graded Lie algebra such that a given representation of a reductive Lie algebra can be embedded into its local part.
To prove Theorem \ref {th;staplie}, the author constructed graded components $U_{0}, U_{\pm 1}, U_{\pm 2}\ldots $ inductively.
\par 
On the other hand, H.Rubenthaler have obtained similar result independently in \cite {Ru}.
In \cite [Theorem 3.1.2]{Ru}, he constructed a local Lie algebra $\Gamma (\g _0,B_0,\pi )$ from a fundamental triplet $(\g_0,B_0,(\pi,U))$, consists of a quadratic Lie algebra $(\g_0,B_0)$ and its finite-dimensional representation $(\pi,U)$, and constructed a graded Lie algebra $\g _{min}(\Gamma (\g_0,B_0,\pi ))$ using \cite [Proposition 4]{ka-1} due to V.Kac.
Although the constructions of him and of the author are based on different theories, their goals coincide.
\begin{theo}\label {th;stapfund}
Let $(\g, \pi ,U,U^*,B)$ be a standard pentad with finite-dimensional objects and a symmetric bilinear form $B$.
Then the corresponding Lie algebra $L(\g,\pi,U,U^*,B)$ is isomorphic to the Lie algebra $\g _{min}(\Gamma (\g ,B,\pi ))$ as graded Lie algebras.
\end{theo}
\begin{proof}
We can check our claim by a similar argument in \cite [p.1278, Proposition 5]{ka-1}.
Suppose that $L(\g,\pi,U,U^*,B)=\bigoplus _{n\in \Z}U_n$ is not minimal (for detail on minimal Lie algebras, see \cite [Definition 6]{ka-1}).
Then there exists a non-zero graded ideal $J=\bigoplus _{n\in \Z}(J\cap U_n)$ such that $J\cap (U_{-1}\oplus U_0\oplus U_1)=J\cap (U^*\oplus \g\oplus U)=\{0\}$.
Take an integer $k$ such that $J\cap U_k\neq \{0\}$ and $J\cap U_n=\{0\}$ for any $|n|<|k|$.
If $k>0$, there exists a non-zero element $v\in J\cap U_n$ such that $[v,U_{-1}]~\{0\}$.
It contradicts to the construction due to the author that $U_k\subset \Hom (U_{-1},U_{k-1})$ (see \cite [Definition 2.9]{Sa3}).
The case where $k<0$ is similar.
\end{proof}
Thus, we can summarize results on the relation between the theory of graded Lie algebras and the theory of representation theory due to the author and due to H.Rubenthaler.
For example, on the structure of graded Lie algebras constructed by the author and by H.Rubenthaler, we have the following result.
\begin{theo}[\text {\cite [Theorem 3.2]{Sa5}}]
Let $\g $ be a finite-dimensional reductive Lie algebra and $(\pi ,U)$ its finite-dimensional completely reducible representation.
Then for any non-degenerate invariant symmetric bilinear form $B$ on $\g $, we have isomorphisms of Lie algebras (up to grading):
\begin{align}
L(\g, \pi ,U,U^*,B)\simeq \g _{min}(\Gamma (\g _0,B,\pi )) \simeq G^{\prime }(C)\oplus \mathfrak{z}\oplus \Delta,\label {iso;1}
\end{align}
where $C$ is a symmetrizable square matrix and $G^{\prime }(C)$ is the reduced contragredient Lie algebra with Cartan matrix $C$.
The Lie algebraic structure of the right hand side of (\ref {iso;1}) is defined by
\begin{align*}
&[G^{\prime }(C)\oplus \mathfrak{z}\oplus \Delta ,G^{\prime }(C)\oplus \mathfrak{z}\oplus \Delta ]=G^{\prime }(C)\oplus \mathfrak{z},\\
&[\mathfrak{z},G^{\prime }(C)\oplus \mathfrak{z}\oplus \Delta ]=[\Delta ,\Delta ]=\{0\},\\
&\text {the adjoint representation of $\Delta $ on $G^{\prime }\oplus \mathfrak{z}\oplus \Delta $ is diagonalizable}.
\end{align*}
\end{theo}
On the other hand, H.Rubenthaler obtained important results on relative invariant in \cite [section 4]{Ru}.
Remain of this paper, we shall use notion and notations based on the author's works unless noted otherwise.
Here, we need to import some notion by H.Rubenthaler to the theory standard pentads.
\begin{defn}[\text {grading element, \cite [Remark 3.4.4]{Ru}}]
Let $(\g,\pi,U,\U,B)$ be a standard pentad.
When an element $H_0\in \g $ satisfies the following conditions:
$$
[H_0,A]=0,\quad [H_0,X]=2X,\quad [H_0,Y]=-2Y
$$
for any $A\in \g$, $X\in U$, $Y\in \U$, we say that $H$ is a grading element of the pentad $(\g,\pi,U,\U,B)$ or the graded Lie algebra $L(\g,\pi,U,\U,B)$ or its local part $\U \oplus \g \oplus U$.
\end{defn}
\begin{defn}[\text {$\sl _2$-triplets, \cite [p.53]{Ru}}]
Let $(\g,\pi,U,\U,B)$ be a standard pentad.
When a triple $(y,h,x)\in L(\g,\pi,U,\U,B)^3$ satisfies the following conditions, we say that the pentad is an $\sl _2$-triple:
$$
[h,x]=2x,\quad [h,y]=-2y,\quad [x,y]=h.
$$
\end{defn}
We give the notion of prehomogeneity of standard pentads.
\begin{defn}[\text {\cite [Definition 2.2]{Sa2}}]
Let $(\g,\pi,U,\U,B)$ be a standard pentad with $\Phi $-map $\Phi _{\pi }$.
When the pentad satisfies the following condition, we say that the pentad $(\g,\pi,U,\U,B)$ is a prehomogeneous pentad:
\begin{itemize}
\item {there exists an element $X\in U$ such that a linear map $\Phi _{\pi }(X\otimes \cdot )\colon \U\rightarrow \g$ defined by $\phi \in \U\mapsto \Phi _{\pi }(X\otimes\phi )\in \g$ is injective.}
\end{itemize}
In other words, a pentad $(\g,\pi,U,\U,B)$ is prehomogeneous if and only if its corresponding Lie algebra $L(\g,\pi,U,\U,B)=\bigoplus _{n\in \Z}U_n$ has an element $X\in U_1$ such that the adjoint map $\ad X\colon U_{-1}\rightarrow U_0$ is injective.
Moreover, we call such an element a generic point of the pentad.
\end{defn}
The terms ``prehomogeneous'' and ``generic points''come from prehomogeneous vector spaces.
\begin{defn}[\text {see, for example \cite [p.35, Definition 1]{SK}}]
Let $G$ be a connected linear algebraic group and $(\rho,V)$ its finite-dimensional rational representation.
We call a triplet $(G,\rho ,V)$ a prehomogeneous vector space (abbrev. PV) when there exists a Zariski-dense orbit $\rho (G)x$ in $V$.
In particular, when a PV $(G,\rho,V)$ has a reductive group $G$, we call it a reductive PV.
An element $x^{\prime }\in V$ is called a generic point when it belongs to the Zariski-dense orbit $\rho (G)x$.
\end{defn}
\begin{theo}[\text {\cite [Theorems 2.1, 2.4]{Sa2}}]
Let $G$ be a finite-dimensional reductive algebraic group and $\g =\Lie (G)$ its Lie algebra.
Let $(\rho ,V)$ be a finite-dimensional representation of $G$ and $(d\rho ,V)$ its infinitesimal representation of $\g $.
Then the following two conditions are equivalent:
\begin{enumerate}
\item {A triplet $(G,\rho,V)$ is a PV.}
\item {A pentad $(\g,d\rho,V,V^*,B)$ is a prehomogeneous pentad for any non-degenerate invariant symmetric bilinear form $B$.}
\end{enumerate}
Moreover, an element $x\in V$ is a generic point of $(G,\rho,V)$ in the sense of PVs if and only if $x$ is a generic point of $(\Lie (G),d\rho,V,V^*,B)$ in the sense of prehomogeneous pentads.
\end{theo}
The theory of PVs have the notion of the regularity.
\begin{defn}[quasi-regular and regular, \text {\cite [p.119]{SKKO}}]\label {defn;regularPV}
Let $(G,\rho,V)$ a PV with a generic point $x$.
Let $G_x=\{g\in G\mid \rho (g)x=x\}$ the isotropy subgroup of $G$ at $x$.
Let $\g _1$ be a subalgebra of $\Lie (G)$ generated by $\Lie (G_x)$ and $[\Lie (G),\Lie (G)]$ and put $\overline {X}_1=\{\omega \in \g ^*=\Hom (\g,\C)\mid \omega \mid _{\g _1}=0\}$.
Then the PV $(G,\rho,V)$ is called quasi-regular if there exists $\omega \in \overline {X}_1$ and a rational map $\varphi _{\omega }:\rho (G)x\rightarrow V^*$ such that 
$$
\varphi _{\omega }(\rho (g)x ^{\prime })=\rho (g)^*\varphi _{\omega }(x ^{\prime }),\quad \langle d\rho (A)x^{\prime },\varphi _{\omega }(x ^{\prime })\rangle =\omega (A)
$$
for any $A\in \Lie (G)$, $g\in G$ and $x ^{\prime }\in \rho (G)x$ and that the image of $\varphi _{\omega }$ is Zariski-dense in $V^*$.
In this case, $\omega $ is called non-degenerate.
In particular, if there exists a character $\chi :G\rightarrow \C$ which corresponds to some relative invariant such that $\omega =d\chi $, then the PV $(G,\rho ,V)$ is called regular.
\end{defn}
In general, we need to distinguish the notions of regular and quasi-regular.
However, under the assumption that a group in a triplet is regular, we have the following theorem.
\begin{theo}[\text {\cite [Proposition 1.3]{SKKO}}]
Let $(G,\rho,V)$ be a PV and assume that $G$ is reductive.
Then $(G,\rho,V)$ is regular if and only if it is quasi-regular.
\end{theo}

\section {PVs and graded Lie algebras}
Remain of this paper, we shall consider how to describe the regularity of PVs using the theory of graded Lie algebras.
In \cite {Ru}, H.Rubenthaler defined the following condition $(P)_x$:
$$
\text {$(P)_x\colon $\quad $x\not \in [[\g,\g ],x]$}
$$
and proved that the condition $(P)_x$ is closely related to $\sl _2$-triplets and relative invariants of a representation (see \cite [pp.53--58, section 4]{Ru}) under the Assumption $(H)$.
\begin{defn}[\text {Assumption $(H)$, \cite [p.53]{Ru}}]
We say that a representation $(\g,\pi,U)$ of a Lie algebra $\g $ satisfies Assumption $(H)$ when the followings hold:
\begin{enumerate}
\item {The Lie algebra $\g $ is a reductive Lie algebra with one-dimensional center: $$\g =Z(\g )\oplus [\g,\g],\quad \dim Z(\g )=1,$$}
\item {We suppose also that $Z(\g)$ acts by a non-trivial character (i.e. $\pi (Z(\g ))=\C \mathrm {Id}$)}.
\end{enumerate}
\end{defn}
If $(\Lie (G),d\rho ,V)$ satisfies the Assumption $(H)$, it means that $(G,\rho,V)$ is a group representation of a reductive group $G$ with $1$-dimensional scalar multiplication.
\par 
Remain of this paper, we shall define similar conditions and consider relations between these conditions and the regularity of PVs.
\begin{defn}
Let $(\g,\pi ,U,\U,B)$ be a standard pentad.
When elements $H\in \g$ and $X\in U$ (respectively $H\in \g$ and $Y\in U^*$) have an element $\eta \in U^*$ (respectively $\xi \in U$) such that a triple $(\eta ,H,X)$ (respectively $(Y,H,\xi $) is an $\sl _2$-triple, we denote that $(P)_{(\cdot ,H,X)}$ (respectively, $(P)_{(Y,H,\cdot )}$).
Moreover, if an element $\eta $ (respectively, $\xi $) is determined from $H$ and $X$ (respectively, $H$ and $Y$) uniquely, we denote that $(P)^!_{(\cdot ,H,X)}$ (respectively, $(P)^!_{(Y,H,\cdot )}$).
\end{defn}
\begin{pr}
Let $(\g,\pi,U,\U,B)$ be a standard pentad.
If there exists elements $H\in \g$ and $X\in U$ satisfying $(P)^!_{(\cdot ,H,X)}$, then the pentad $(\g,\pi,U,\U,B)$ is prehomogeneous with generic point $X$.
\end{pr}
\begin{proof}
Take a unique element $Y=Y(H,X)$ such that $(Y,H,X)$ is an $\sl _2$-triple.
If we suppose that $X$ is not a generic point, there exists a non-zero element $\eta \in \U$ such that $[X,\eta ]=0$.
Then we have two $\sl _2$-triples $(Y,H,X)$ and $(Y+\eta ,H,X)$, of course $Y\neq Y+\eta $.
It contradicts to the assumption $(P)^!_{(\cdot ,H,X)}$.
\end{proof}
\begin{col}
Let $(G,\rho,V)$ be a triplet with reductive group $G$.
If a pentad $(\Lie (G),d\rho,V,V^*,B)$ has elements $H\in \Lie (G)$ and $X\in V$ satisfying $(P)^!_{(\cdot ,H,X)}$, then the triplet $(G,\rho,V)$ is a PV.
\end{col}
Similarly, we have the following claim.
\begin{pr}
Let $(\g,\rho,V,V^*,B)$ be a prehomogeneous pentad and $X\in V$ be a generic point of it.
If there exists an element $H\in \g$ satisfying $(P)_{(\cdot ,H,X)}$, then $H$ and $X$ satisfy $(P)^!_{(\cdot ,H,X)}$.
\end{pr}
Under these notations, we have the main theorem of this paper.
\begin{theo}
Let $(G,\rho,V)$ be a triplet with a reductive group $G$.
Assume that $(\Lie (G),d\rho,V)$ satisfies the Assumption $(H)$.
Then the following conditions are equivalent:
\begin{enumerate}
\item {The triplet $(G,\rho ,V)$ is a regular PV},
\item {For any non-degenerate invariant bilinear form $B$ on $\Lie (G)$, a pentad $(\Lie (G),d\rho ,V,V^*,B)$ has elements $X\in V$ and $Y\in V^*$ such that $(Y,H_0,X)$ is an $\sl _2$-triple and that the conditions $(P)^!_{(\cdot ,H_0,X)}$ and $(P)^!_{(Y,H_0,\cdot )}$ hold, where $H_0$ is a grading element of the pentad}.
\end{enumerate}
\end{theo}
\begin{proof}
\par
(1. implies 2.)
\par 
We assume that the triplet $(G,\rho,V)$ is a regular PV with a generic point $x\in V$.
Then there exists a non-degenerate linear map $\omega :Z(\Lie (G))\rightarrow \C$ and a rational map $\varphi _{\omega }:\rho (G)x\rightarrow V^*$ such that 
$$
\varphi _{\omega }(\rho (g)x^{\prime })=\rho ^*(g)\varphi _{\omega }(x^{\prime }),
\quad
\langle d\rho (A)x^{\prime }, \varphi _{\omega }(x^{\prime }) \rangle =\omega (A)
$$
for any $g\in G$, $A\in \Lie (G)$ and $x^{\prime }\in \rho (G)x$.
From the assumption that $B$ is a non-degenerate bilinear form, there exists an element $H\in \Lie (G)$ such that $\omega (A)=B(A,H)$.
Since $\omega \mid _{[\Lie (G),\Lie (G)]}=0$, $H=c H_0\in Z(\Lie (G))$.
Here, since $\omega $ is non-degenerate, $c\neq 0$.
Thus, in $L(\Lie (G),d\rho,V,V^*,B)$, we have an equation 
$$
B(A,[x^{\prime }, \varphi _{\omega }(x^{\prime })])=\langle d\rho (A)x^{\prime },\varphi _{\omega }(x^{\prime })\rangle =\omega (A)=B(A,H)
$$
for any $A\in \Lie (G)$ and $x^{\prime }\in \rho (G)x$.
From this, we can deduce that 
$$
[x,\varphi _{\omega }(x)]=H=cH_0.
$$
Then, $(Y,H_0,X)=((1/c)\varphi _{\omega }(x),H_0,x)$ is an $\sl _2$-triple.
Since $X$ belongs to the Zariski-dense orbit, we have $(P)^!_{(\cdot ,H_0,X)}$.
Moreover, since the orbit $\rho ^*(G)Y=(\text {the image of $\varphi _{\omega }$})$ is Zariski-dense in $V^*$, we have $(P)^!_{(Y,H_0,\cdot )}$.
Thus we have the condition 2.
\par 
(2. implies 1.)
\par 
We suppose the condition 2 and take an arbitrary non-degenerate invariant bilinear form $B$.
Then $(G,\rho ,V)$ is a PV with a Zariski-dense orbit $\rho (G)X\subset V$.
We can define a map $\varphi :\rho (G)X\rightarrow V^*$ by
$$
\text {$(x^{\prime },H_0,\varphi (x^{\prime }))$ is an $\sl _2$-triple for any $x^{\prime }\in \rho (G)X$}
$$
satisfying 
$$
\rho ^*(g)\varphi (x^{\prime })=\varphi (\rho (g)x^{\prime })\quad (g\in G,x^{\prime }\in \rho (G)X)
$$
(see \cite [proof of Proposition 4.2.7]{Ru}).
That is, $\eta =\varphi (x^{\prime })$ is a unique solution of a linear equation $\ad (x^{\prime })\eta =H_0$.
Thus, $\varphi $ is a rational map.
If we define $\omega :\Lie (G)\rightarrow \C$ by $\omega (A)=B(A,H_0)$, then we have an equation 
$$
\langle d\rho (A)x^{\prime },\varphi (x^{\prime }) \rangle =B(A,H_0)=\omega (A).
$$
In the notations of Definition \ref {defn;regularPV}, $\omega $ clearly belongs to $\overline {X}_1$.
From the assumption $(P)^!_{(Y,H_0,\cdot )}$, we have that $(\text {the image of $\varphi $})=\rho ^*(G)\varphi (X)=\rho ^*(G)Y$ is Zariski-dense in $V^*$.
Thus, we have that $(G,\rho,V)$ is quasi-regular, and thus, regular.
\end{proof}
\begin{defn}
We define a bilinear form $T_n$ on $\gl _n$ by 
$$
T_n(X,Y)=\Tr (XY)
$$
for any $X,Y\in \gl _n$.
Clearly, the bilinear form $T_n$ is non-degenerate and invariant.
Moreover, for a Lie subalgebra $\l\subset \gl _n$, we also denote the restriction $T_n\mid _{\l \times \l }$ by the same symbol $T_n$.
\end{defn}
\begin{ex}[\text {cf. \cite [pp.103--105]{SK}}]
An irreducible PV $(G,\rho,V)=(GL_1\times Sp _{n}\times SO_3,\Box \otimes \Lambda _1\otimes \Lambda _1, M(2n,3))$ $(2n\geq 3)$ is important since it is an example of a non-regular PV which has a relative invariant (see \cite [Proposition 19, p.105]{SK}).
Let us show this using a pentad $(\Lie (G),d\rho,V,V^*,B)=(\gl _1\oplus \sp _{n}\oplus \so _2, \Box \otimes \Lambda _1\otimes \Lambda _1, M(2n,3), M(2n,3), \Tr _1\oplus \Tr _{2n}\oplus \Tr_2)$, where 
$$
\sp _n=\{A\in M(2n,2n)\mid A\cdot J_n+J_n\cdot {}^tA=O_{2n}\},\quad J_n=\left (\begin{array}{c|c}O_n&I_n\\ \hline -I_n&O_n\end{array}\right ).
$$
The representations $d\rho =\Box\otimes \Lambda _1\otimes \Lambda _1$ and its dual of $\Lie (G)=\gl _1\oplus \sp _{n}\oplus \so _3$ is given by:
\begin{align*}
d\rho (a,A,B)v=av+Av-vB,\quad
d\rho ^*(a,A,B)u=-au+Au-uB, \quad (a\in \gl _1,A\in \sp _{n},B\in \so _3)
\end{align*}
via a bilinear form 
$$
\langle v,u\rangle =\Tr ({}^t v\cdot J_{n}\cdot u) \quad (v,u\in M(2n,3)).
$$
The $\Phi $-map $\Phi _{d\rho }$ of this pentad is given by
$$
\Phi _{d\rho }(v\otimes u)=\left (
\Tr ({}^t v\cdot J_{n}\cdot u) ,\quad -\frac{1}{2}\left (v\cdot {}^t u+u\cdot {}^t v\right )J_n,\quad \frac{1}{2}\left ({}^t v\cdot J_n\cdot u+{}^t u\cdot J_n\cdot v\right )
\right )
$$
We can easily check that the pentad has a grading element $H_0$ and a generic point $X$
$$
H_0=(2,O_{2n},O_3)\in \Lie (G),\quad 
X=
\left (
\begin{array}{c}
\begin{array}{ccc}
1&0&0\\0&1&0
\end{array}\\ \hline 
O_{n-2,3}\\ \hline 
\begin{array}{ccc}
0&0&1 
\end{array}\\ \hline
O_{n-1,3}
\end{array}
\right )\in V=M(2n,3)
$$
and satisfies the Assumption $(H)$.
We can easily check that $H_0$ and $X$ satisfy $(P)^!_{(\cdot ,H_0,X)}$.
In fact, if we put 
$$
\eta =\left (\begin{array}{c}\begin{array}{ccc}0&0&-1\end{array}\\ \hline O_{n-1,3}\\ \hline \begin{array}{ccc}1&0&0\end{array}\\ \hline O_{n-1,3}\end{array}\right ), 
$$
then we have an $\sl _2$-triple $(\eta ,H_0,X)$, and thus $(P)_{(\cdot ,H_0,X)}$.
Since $X$ is a generic point, we have $(P)^!_{(\cdot ,H_0,X)}$.
From the result of H.Rubenthaler, \cite [Theorem 4.2.3]{Ru}, we can deduce that there exists a non-trivial relative invariant on $V=(\text {the Zariski closure of $\rho (G)X$})$.
However, $(G,\rho,V)$ is not regular.
If we suppose that $(G,\rho,V)$ is regular, then we have an $\sl _2$-triple $(Y^{\prime },H_0,X^{\prime })$ such that $(P)^!_{(\cdot ,H_0,X^{\prime })}$ and $(P)^!_{(Y^{\prime },H_0,\cdot )}$ hold.
Then, there exists $g\in G$ such that $X^{\prime }=\rho (g)X$.
Then, we have that $\eta =\rho ^*(g^{-1})Y^{\prime }$ belongs to the Zariski-dense orbit $\rho ^*(G)Y^{\prime }$ in $V^*$.
However, since $\rank \eta =2$, the orbit $\rho ^*(G)\eta $ cannot be Zariski-dense in $M(2n,3)$.
It is a contradiction.
\end{ex}

\medskip
\begin{flushleft}
Nagatoshi Sasano\\
Institute of Mathematics-for-Industry\\
Kyushu University\\
744, Motooka, Nishi-ku, Fukuoka 819-0395\\
Japan\\
E-mail: n-sasano@math.kyushu-u.ac.jp
\end{flushleft}

\end{document}